\documentclass[12pt]{article}
\title{Global actions, $K$-theory and unimodular rows}
\author{Anthony Bak \& Anuradha S. Garge \\
\small 
Fakul{\"a}t f{\"u}r Mathematik, Universit{\"a}t Bielefeld, \\
\small
Bielefeld - 33501, Germany. \\ \small \&\\
\small
Department of Mathematics, \\
\small 
University of Mumbai, Mumbai- 400098, India.}

\date{{\small \today}} 

\usepackage{amsmath,amssymb,amsfonts}
\usepackage{amsthm}
\usepackage{graphicx,latexsym}

\usepackage{fancyhdr}
\usepackage{graphicx}
\usepackage{amscd, amsmath,amssymb,amsfonts}
\usepackage{amsthm} 
\usepackage{ulem}
\usepackage[all, cmtip]{xy} 
\usepackage{fancyhdr}
\usepackage{graphicx}
\usepackage{amscd, amsmath,amssymb,amsfonts}
\usepackage{amsthm}
\usepackage{ulem}

\newtheorem{thm}{Theorem}[section]
\theoremstyle{definition}
\newtheorem{defn}[thm]{Definition}
\theoremstyle{remark}

\theoremstyle{plain}
\newtheorem{lem}[thm]{Lemma}
\newtheorem{cor}[thm]{Corollary}
\newtheorem{prop}[thm]{Proposition}
\newtheorem{rem}[thm]{Remark}

\newtheorem{exa}[thm]{Example}
\theoremstyle{plain}

\def\B{{\operatorname{B}}}            
\def\E{{\operatorname{E}}}            
\def\EUm{{\operatorname{EUm}}}        
\def\GL{{\operatorname{GL}}}          
\def\K{{\operatorname{K}}}            
\def\P{{\operatorname{P}}}            
\def\EP{{\operatorname{EP}}}          
\def\StUm{{\operatorname{StUm}}}         
\def\SL{{\operatorname{SL}}}          
\def\St{{\operatorname{St}}}          
\def\Um{{\operatorname{Um}}}          
 
\def\Z{{\mathbb Z}}                   

\begin{document}
\maketitle{}

\begin{abstract}
Global actions were introduced by Bak \cite{Bak} in order to have a homotopy theory 
in a purely algebraic setting. In this paper we apply his techniques in a particular case: 
the (single domain) unimodular row global action. More precisely, we compute the 
the path connected component and fundamental group for the unimodular row global action. 
An explicit computation of the fundamental group of the (connected component of) unimodular row global action 
is closely related to stability questions in $\K$-theory. This will be shown by constructing an exact sequence  
with the fundamental group functor as the middle term and having surjective stability for the functor $\K_2$ on the left 
and injective stability for the functor $\K_1$ on the right.
\end{abstract}


\section{Introduction} 

Global actions were introduced by Bak, (see \cite{Bak}) in order to have the 
flexibility of combining algebraic and topological ideas. In this paper, we will 
be concentrating on single domain global actions. A single domain global action consists of 
the following data: a set together with several groups acting on it such that these group actions 
satisfy a certain compatibility condition. There is a well-defined notion of homotopy in this setting and 
here we show that this circle of ideas can be applied very effectively in the following situation: 
given an associative ring $R$ with unity, a unimodular row of length $n$ over $R$ 
is by definition, an $n$-tuple of the form $v = (v_1, \ldots, v_n)$, with $v_i \in R$ 
such that there exists $w = (w_1, \ldots, w_n)$, $w_i \in R$ with $v \cdot w^t = \langle v, w \rangle := 
\sum_i v_iw_i = 1$. The set of all unimodular rows of length $n$ over $R$ is denoted by $\Um_n(R)$. 
The action of the general linear group $\GL_n(R)$ (and hence its elementary subgroup 
$\E_n(R)$) on $\Um_n(R)$ allows one to define in a natural way a single domain global action structure
on $\Um_n(R)$.

 The aim of this paper is to investigate $\pi_0$ and $\pi_1$ of this single domain global action, 
and to show that both objects are closely related to stability questions in algebraic $K$-theory.
 An algebraic description for $\pi_0$ is easy to formulate and prove: 
                   $$\pi_0(Um_n(R)) = Um_n(R)/E_n(R),$$ 
where the object on the right denotes the orbit space of the action of $\E_n(R)$ on $\Um_n(R)$ with base point 
the orbit of $e_1=(1,0,...,0)$ in $\Um_n(R).$

   An algebraic description of $\pi_1(Um_n(R))$ is more difficult and can defined for the connected component of the base point 
of this global action, which we denote by $\EUm_n(R).$  
Let $\St_n(R)$ denote the Steinberg group and let 
$\theta_n: \St_n(R) \rightarrow \GL_n(R)$ denote the standard group homomorphism defined by sending generators
$X_{ij}(r)$ ($i \neq j$) of the Steinberg group to the elementary matrices $E_{ij}(r)$ ($i \neq j$). 
(Note here that for $i \neq j$ if $e_{ij}$ denotes the $n \times n$ matrix whose $(i, j)$-th entry is $1$ and all 
other entries are $0$, then for $r \in R$, let $E_{ij}(r)= I_n + re_{ij}$, where $I_n$ denotes the $n \times n$ 
identity matrix.) Set: 
\begin{itemize}
\item[1.] $\P_n(R)= \{\sigma \in \GL_n(R)| e_1\sigma = e_1\}$.

\item[2.] $\EP_n(R)= P_n(R) \cap \E_n(R)$.

\item[3.] $\widetilde{P_n}(R)$ {~\rm preimage of~} $\EP_n(R)~{\rm in}~\St_n(R):= \theta_n^{-1}(\EP_n(R))$.

\item[4.] $\B_n(R)$ is a certain normal subgroup of $\widetilde{P_n}(R)$ to be defined in Section $\S 3.$ 
\end{itemize} 

Then, one has $$\pi_1(\EUm_n(R)) = \widetilde{P_n}(R)/\B_n(R).$$ 

  The relationship of $\pi_0$ and $\pi_1$ of $\Um_n(R)$ to stability in algebraic $\K$-theory is expressed by two 
short exact sequences, one with $\pi_1$ as its middle term and the other with $\pi_0$ as its middle term. The sequence 
with $\pi_1$ has surjective stability for the functor $\K_2$ on the left and injective stability for the functor 
$\K_1$ on the right. The sequence with $\pi_0$ has surjective stabilty for the functor $\K_1$ on the left and 
injective stability for the functor $\K_0$ on the right. Together,these short exact sequences are equivalent to 
the $8$-term exact sequence: 
$$(\K_{2,n}(R))_2 \rightarrow \K_{2,n}(R) \rightarrow \pi_1(\EUm_n(R)) \rightarrow \K_{1,n-1}(R)/(\K_{1,n-1}(R))_2  
\rightarrow $$ $$\K_{1,n}(R) \rightarrow  \pi_0(\Um_n(R)) \rightarrow \K^s_{0,n-1}(R) \rightarrow \K^s_{0,n}(R).$$
of pointed sets. By definition, $$(\K_{2,n}(R))_2 = \K_{2,n}(R) \cap (\widetilde{P_n}(R))_2$$ 
and contains the $${\rm image}(\K_{2,n-1}(R) \rightarrow \K_{2,n}(R))$$ (see (\S 5));
            $$(\K_{1,n-1}(R))_2 = (\GL_{n-1}(R) \cap (\EP_{n}(R))_2)/\E_{n-1}(R),$$
where $(\EP_n(R))_2$ is a normal subgroup of $\EP_n(R)$ (see (\S 5)), which contains $\E_{n-1}(R)$. 
Let $\K^s_{0,m}(R)$ be the set of all isomorphism classes of finitely generated projective modules $P$ such that for some 
$r$ (depending on $P$) $P \oplus R^r = R^{m+r}$. The base point of $K^s_{0,m}(R)$ is the isomorphism class of $R^m$.

  The first $3$ terms of the exact sequence above, starting from the left, come equipped with group structures and 
the maps between them are group homomorphisms. So this much of the sequence is an exact sequence of groups.  
Suppose that $\E_{n-1}(R)$ and $\E_n(R)$ are normal in $\GL_{n-1}(R)$ and $\GL_n(R)$ respectively. Then $\K_{1,n-1}(R)$ 
and $\K_{1,n}(R)$ are groups and it turns out that $(\K_{1,n-1}(R))_2$ is a normal subgroup of 
$\rm{ker}(\K_{1,n-1}(R) \rightarrow \K_{1,n}(R))$ and that the map 
$\pi_1(\Um_n(R)) \rightarrow \K_{1,n-1}(R)/(\K_{1,n-1}(R))_2$ has as image the group
$[{\rm ker}(\K_{1,n-1}(R) \rightarrow \K_{1,n}(R))]/(\K_{1,n-1}(R))_2$ and is a group homomorphism to this group. 
So in this case, the first $5$ terms behave like an exact sequence of groups. 
It is an interesting problem to find group structures
on the remaining objects so that the entire sequence behaves like an exact sequence of groups. 

Assuming the ring $R$ is commutative and noetherian of finite Krull dimension $d$ and $n$ is sufficiently large 
relative to $d$, van der Kallen \cite{vdK1}, \cite{vdK2} has found a group structure on $\pi_0(\Um_n(R))$, but 
has shown that the map $\GL_n(R) \rightarrow \pi_0(\Um_n(R))$ is not always a group homomorphism. 
On the other hand, Ravi Rao and van der Kallen \cite{Rao-vdk} have found (nontrivial) examples where it is a 
group homomorphism. In these examples, we get a $6$-term sequence which behaves like an exact sequence of groups. 
An interesting problem is to find group structures on the $\K^s_{0,i}(R)$
such that the maps involving these groups in the sequence are group homomorphisms. 

The rest of the paper is organized as follows: Section $2$ gives basic defintions and many relevant 
examples of global actions. Section $3$ describes the notion of homotopy for global actions and in Section $4$
we give the details on simply connected coverings of global actions: first a global-action theoretic construction
and then an algebraic one. Universality of the simply connected covering then implies that these two constructions
are isomorphic. With this one computes the fundamental group of the elementary unimodular row global action.  
Section $5$ constructs the exact sequence mentioned in the introduction and deduces an interesting 
corollary on the vanishing of the fundamental group, as predicted by algebra. 

\section{Preliminaries}

\subsection{Global actions} 

In this section, we recall from \cite{Bak} the definition of a global action, a single 
domain global action and their morphisms, and  provide some examples. 
We begin with the definition of a group acting on a set. 

\begin{defn}
If $G$ is a group and $X$ is a set, then a (right) {\it group action} of $G$ on $X$
is a function $X \times G \rightarrow X$, denoted by  
$(x, g) \mapsto x \cdot g$, such that:
\begin{itemize}
\item[1.] $x \cdot e = x$, for all $x \in X$, where $e$ is the identity
of the group $G$. 
\item[2.] $x \cdot (g_1 g_2) = (x \cdot g_1) \cdot g_2$, for all 
$x \in X$ and $g_1, g_2 \in G$. 
\end{itemize}

Such a group action will be denoted by $X \curvearrowleft G$. 
\end{defn}

\begin{defn}
Let $X, Y$ be sets with groups $G, H$ acting on them respectively. 
A {\it morphism of group actions}, $(\psi, \varphi): X \curvearrowleft 
G \rightarrow Y \curvearrowleft H$,
consists of a function $\psi: X \rightarrow Y$ and a homomorphism of groups 
$\varphi:G \rightarrow H$  
such that $\psi(x \cdot g) = \psi(x) \cdot \varphi(g)$. 
\end{defn}

\begin{defn}\label{defn:global}
A {\it global action} $A$ consists of a set $X_A$ (called the underlying set of $A$) together with:
\begin{itemize}
\item[1.] An indexing set $\Phi_A$, having a reflexive relation $\leq$ on it. 
\item[2.] A family $\{(X_A)_{\alpha} \curvearrowleft (G_{A})_{\alpha}~|~\alpha \in \Phi_A \}$ of 
group actions on subsets $(X_A)_{\alpha}$ of $X_A$. The $(G_{A})_{\alpha}$ are called the 
local groups of the global action. 
\item[3.] For each pair $\alpha \leq \beta$ in $\Phi_A$, a group homomorphism, 
$$(\theta_{A})_{\alpha \beta}:(G_{A})_{\alpha} \rightarrow (G_{A})_{\beta},$$  called a structure 
homomorphism such that: 
\begin{itemize}
\item[(a)]The groups $(G_{A})_{\alpha}$ 
leave $(X_A)_{\alpha} \cap (X_A)_{\beta}$ invariant. 
\item[(b)] The pair   
$$({\rm inclusion}, ({\theta_A})_{\alpha \beta}): ((X_A)_{\alpha} \cap (X_A)_{\beta}) \curvearrowleft (G_{A})_{\alpha}
\rightarrow (X_A)_{\beta} \curvearrowleft (G_{A})_{\beta}$$ 
is a morphism of group actions. (This will be called the compatibility condition). 
\end{itemize}
\end{itemize}
\end{defn}

\begin{defn}
A global action $A$ is said to be a {\it single domain global action} if 
$(X_A)_{\alpha} = X_A$, for all $\alpha \in \Phi_A$. 
\end{defn}

\begin{rem}
For simplicity of notation whenever only one global action is involved, 
we shall drop the suffix $A$ everywhere in the definition and 
write $X, G_{\alpha}, X_{\alpha}, \theta_{\alpha \beta}$ instead. 
\end{rem}

\begin{defn}\label{defn:quotient} 
Let $G$ be a group and let $\Phi$ be an index set (equipped with a reflexive relation) 
for a family $G_{\alpha}$ $(\alpha \in \Phi)$
of subgroups of $G$. One defines a single domain global action $A$ from this data, by letting 
$X = G$ and letting each $G_{\alpha}$ act on $X$ by right multiplication. If $H$ denotes 
a subgroup of $G$, then one can make the space $G/H$ of right cosets $Hg$ of $H$ in $G$ into 
a single domain global action by letting each $G_{\alpha}$ act on $G/H$ in the obvious way, i.e.
$(Hg)g_{\alpha} = Hgg_{\alpha}$, for all $g_{\alpha} \in G_{\alpha}$. 

\end{defn} 

We recall the definition of a morphism between global actions from \cite{Bak}.
To do this one requires the notion of a local frame, which is defined below. 

\begin{defn}
Let $A$ be a global action. Let $x \in X_{\alpha}$ be some point in 
a local set of $A$. A {\it local frame} at $x$ in $\alpha$ or an $\alpha$-frame at $x$ is a 
finite subset, say $\{ x = x_0, \ldots, x_p \}$ of $X_{\alpha}$ such that 
$G_{\alpha}$-action on $X_{\alpha}$ is transitive on 
$x_0, \ldots, x_p$ i.e., for each $j$, $1 \leq j \leq p$, there exists $g_j \in G_{\alpha}$ such that 
$x_0 \cdot g_j = x_j$. 
\end{defn}

\begin{defn}
If $A$ and $B$ are global actions, with underlying sets $X, Y$ and index sets $\Phi, \Psi$ respectively, 
a {\it morphism of global actions} is a function $f: X \rightarrow Y$
which preserves local frames. We shall denote such a morphism by $f:A \rightarrow B$. 
More precisely, if $x_0, \ldots, x_p$ is an 
$\alpha$-frame at $x = x_0$,
then $f(x_0), \ldots, f(x_p)$ is a $\beta$-frame at 
$f(x) = f(x_0)$ for some $\beta \in \Psi$. 
\end{defn}

\begin{exa}
Let $A$ be a global action. Then the identity function from the underlying set of $A$ to itself 
is a morphism of global actions. 
\end{exa}

\subsection{Important examples of global actions} 

  We give below some examples of global actions by describing their underlying set, indexing set, 
local sets and local groups. It is easy to check that the compatibility condition holds. (See \cite{Bak}). 

\begin{itemize}
\item{\bf The line action:} The {\it line action}, denoted by $L$ is a global action with underlying set 
$X = \Z$ and indexing set $\Phi = \Z \cup \{*\}$. Let the only relations in $\Phi$ be $* \leq n$, for all $n\in \Z$ and 
$n \leq n$ for all $n \in \Z$. The local sets are $X_n = \{n, n+1\}$ if $n \in \Z$ and $X_{*} = \Z$. 
Let the local groups be $G_n = \Z/2\Z$, if $n \in \Z$, $G_{*} = 1$ and let 
$\{n, n+1 \} \curvearrowleft G_n$ be the group action such that the non-trivial element 
of $G_n$ exchanges the elements $n, n+1$. Let $\theta_{* \leq n}: \{1\} \rightarrow G_n$ denote the unique 
group homomorphism. 

\item{\bf The general linear global action:} 

Given $n \geq 3$, let $J_n = ([1, n] \times [1, n]) \setminus \{ (i,i)~|~1 \leq i \leq n \}$ 
i.e. the cartesian product of the set $\{1, 2, \ldots, n\}$ with itself with the diagonal removed. 

A subset $\alpha \in J_n$ is called {\it nilpotent} if the following conditions hold:
\begin{itemize}
\item If $(i, j) \in \alpha$, then $(j, i) \notin \alpha$. 
\item If $(i, j), (j, k) \in \alpha$, then $(i, k) \in \alpha$. 
\end{itemize}

Note that the empty set is a nilpotent subset and that the intersection of nilpotent subsets is nilpotent. 
Let $R$ denote an associative ring with unity. 
The {\it general linear global action}, which we denote by $\GL_n(R)$ has underlying set $\GL_n(R)$, the general linear group. 
The index set $\Phi_n$ is the set of all nilpotent subsets $\alpha$ of $J_n$. 
We give $\Phi_n$ the partial ordering defined by $\alpha \leq \beta \Leftrightarrow 
\alpha \subseteq \beta$. For all $\alpha \in \Phi_n$, let the local set $(X_{\GL_n(R)})_{\alpha} = \GL_n(R)$. 
For all $\alpha \in \Phi$, let the local group ${\GL_n(R)}_{\alpha}$ be the subgroup of 
$\GL_n(R)$ consisting of all matrices whose diagonal coefficients are 
$1$, whose nondiagonal coefficients are $0$ for coordinates $(i, j) \not \in \alpha$ and arbitrary for coordinates $(i, j) \in \alpha$. 
This means that the empty subset of $\Phi_n$ is assigned the trivial subgroup of $\GL_n(R)$.
Clearly $(\GL_n(R))_{\alpha} \cap (\GL_n(R))_{\beta} = (\GL_n(R))_{\alpha \cap \beta}$. 
Thus the assignment $$\Phi_n \rightarrow {\rm~subgroups~ of ~ \GL_n(R)},$$ sending 
$$\alpha \rightarrow (\GL_n(R))_{\alpha}$$ preserves not only partial orderings i.e. $\alpha \leq \beta \implies 
(\GL_n(R))_{\alpha} \subset (\GL_n(R))_{\beta}$, but also intersections i.e. 
$(\GL_n(R))_{\alpha \cap \beta} = (\GL_n(R))_{\alpha} \cap (\GL_n(R))_{\beta}$. It is straight forward and 
easy to verify that if we assign to each pair $\alpha \leq \beta \in \Phi_n$ the natural inclusion 
$\varphi_{\alpha \beta}: (\GL_n(R))_{\alpha} \rightarrow (\GL_n(R))_{\beta}$ then we get a (single domain) global action.
The intersection property is not needed here. It will be used later to establish the covering property in the sense 
of the Steinberg extension.

It is not difficult to show that $(\GL_n(R))_{\alpha}$ is 
generated by all elementary matrices $E_{ij}(r)$, where $(i, j) \in \alpha$ and $r \in R$. 
We recall the definition of an elementary matrix. If $(i, j) \in J_n$, let 
$e_{ij}$ denote the $n \times n$ matrix whose $(i, j)$-th entry is $1$ and all other 
entries are $0$. For $r \in R$, let $E_{ij}(r)= I_n + re_{ij}$, where $I_n$ denotes the $n \times n$ identity matrix. 

The subgroups $(\GL_n(R))_{\alpha}, \alpha \in \Phi$ are known in the literature as the {\it standard unipotent subgroups} 
of $\GL_n(R)$. Since any elementary matrix is contained in some $(\GL_n(R))_{\alpha}$ and since each 
$(\GL_n(R))_{\alpha}$ is generated by elementary matrices, it follows by definition that the $(\GL_n(R))_{\alpha}$
generate the elementary subgroup $\E_n(R)$ of $\GL_n(R)$.

\item{\bf The elementary global action:} 
 The {\it elementary global action} has underlying set $\E_n(R)$, the elementary group. The indexing set 
as well as the local groups are the same as those for the 
general linear global action. Each local set is the whole $\E_n(R)$.

\item{\bf The special linear global action:} 
Suppose $R$ is commutative. The {\it special linear global action} has underlying set $\SL_n(R)$, the special linear group.  
The indexing set as well as the local groups are the same as those for the 
general linear global action. Each local set is the whole $\SL_n(R)$.  
\end{itemize} 

Abusing notation, we shall let $\GL_n(R), \E_n(R)$ and $\SL_n(R)$ denote repectively the global actions defined above. 
Clearly, the canonical inclusions $\E_n(R) \rightarrow \GL_n(R)$ and when $R$ is commutative, $\E_n(R) \rightarrow \SL_n(R) 
\rightarrow \GL_n(R)$ are morphisms of global actions. 

 Before we begin describing the Steinberg global action, we recall the 
definition of the Steinberg group itself from \cite{Mil}, \S 5. 

Recall that elementary matrices satisfy the property 

\begin{itemize}
\item $E_{ij}(r)E_{ij}(s) = E_{ij}(r+s)$, for all $r, s \in R,$
\end{itemize}

and that the following commutator formulae hold:  

\begin{itemize}
\item $[E_{ij}(r)~~E_{kl}(s)] = 1$, if $j \neq k, i \neq l$, $r, s 
\in R$.   
\item $[E_{ij}(r)~~E_{jl}(s)] = E_{il}(rs)$, if $i \neq l$, $r, s \in R$. 
\end{itemize}

The {\it Steinberg group} $\St_n(R)$, associated to a ring $R$ is the free group defined by the generators 
$X_{ij}(r)$, $r \in R, (1 \leq i, j \leq n, i \neq j)$ subject to exactly the same relations above with $E_{ij}$ replaced by $X_{ij}$. 
Thus, the Steinberg group is defined as a quotient 
$\mathfrak{F}/\mathfrak{N}$, where $\mathfrak{F}$ denotes the free group 
generated by the symbols $X_{ij}(r), r \in R$ and $\mathfrak{N}$ denotes the smallest 
normal subgroup of $\mathfrak{F}$ modulo which the above relations are valid. 
The assignment $X_{ij}(r) \rightarrow E_{ij}(r)$ sends the relations between the generators of $\St_n(R)$ into valid identities  
between elementary matrices.

\begin{itemize} 
\item{\bf The Steinberg global action} 
The Steinberg global action has underlying set $\St_n(R)$, the Steinberg group. 
The indexing set $\Phi_n$ is the same as that of the general linear global action. 
For all $\alpha \in \Phi_n$ the local set $(X_{\St_n(R)})_{\alpha} = \St_n(R)$ and the 
local group $(\St_n(R))_{\alpha} = \langle X_{ij}(r)~|~(i, j) \in \alpha, r \in R \rangle$. 
If we assign to each pair $\alpha \leq \beta \in \Phi_n$ the canonical inclusion $\varphi_{\alpha \beta}: 
(\St_n(R))_{\alpha} \rightarrow (\St_n(R))_{\beta}$ then it is straight forward and easy to 
show that we get a (single domain) global action. 
\end{itemize} 

Abusing notation, we let $\St_n(R)$ denote this global action. Clearly the canonical homomorphism 
of groups $\St_n(R) \rightarrow \E_n(R)$ described above is a morphism of global actions. 

The next proposition provides the algebraic facts about the Steinberg group, which will be needed in the 
(algebraic) homotopy theory of $\GL_n(R)$ and $\Um_n(R)$ (to be defined in the next section). 

\begin{prop}\label{prop:intersection} 
Let $\theta:\St_n(R) \rightarrow \E_n(R)$ denote the canonical homomorphism.
Let $\theta_{E}:\Phi_n \rightarrow {\rm subgroups~of~\E_n(R)}$, $\alpha \rightarrow 
(\E_n(R))_{\alpha}$ and let 
$\theta_{\St}:\Phi_n \rightarrow {\rm subgroups~of~\St_n(R)}$, $\alpha \rightarrow (\St_n(R))_{\alpha}$.  
Clearly, $\theta_{E}$ and $\theta_{\St}$ are partial order preserving maps and thus functors. 
With these notations, one has: 
\begin{itemize}
\item[{\rm[1.]}] The maps  $\theta_{E}$ and $\theta_{\St}$ preserve intersections and the commutative diagram 
$$\xymatrixcolsep{7pc}
\xymatrix{
\Phi_n \ar[r]^{{\theta_\St}} \ar[rd]_{{\theta_E}} & \rm Subgrps.~of~\St_n(R) \ar[d]^\theta \\
& \rm Subgrps.~of~\E_n(R)
}$$
defines a natural isomorphism $\theta_{\St} \rightarrow \theta_{E}$ of functors. 

\item[{\rm[2.]}] The canonical homomorphisms 
$${\rm colim}(\E_n(R))_{\alpha} \longleftarrow {\rm colim}(\St_n(R))_{\alpha} \longrightarrow \St_n(R)$$
are isomorphisms.  
\item[{\rm[3.]}] The canonical map 
$$\bigcup_{\alpha \in \Phi_n} (\St_n(R))_{\alpha} \rightarrow \bigcup_{\alpha \in \Phi_n} (\E_n(R))_{\alpha}$$ is bijective. 
\end{itemize}
\end{prop} 
\begin{proof}
\begin{itemize}
\item[{\rm[1.]}]
Let $S_n$ denote the group of permutations of $n$ elements. 
Let $\pi = \begin{pmatrix}
1  & 2 & \ldots & n \\
\pi(1) & \pi(2) & \ldots & \pi(n) 
\end{pmatrix} \in S_n$. To each element $\pi$, we associate the permutation matrix 
$M_{\pi},$ whose $\pi_i$-the column has zeroes in all entries except $i$-th where it has $1$. 

The groups $S_n$ acts on $\GL_n(R)$ on the right by conjugation by permutation matrices. It is easy to chek that 
$E_{ij}(r)^{\pi} = M_{\pi^{-1}}E_{ij}(r)M_{\pi} = E_{(i\pi) (j\pi)}$ and the resulting action of $S_n$ on 
$\E_n(R)$ preserves the $3$ relations above for elementary matrices. Thus the action of $S_n$ on $\E_n(R)$ lifts to 
an action of $S_n$ on $\St_n(R)$ such that the homomorphism $\theta$ is $S_n$ equivariant. The group $S_n$ acts on 
$J_n$ in the obvious way namely, $(i, j)\pi = (i\pi, j\pi)$ and there is an induced action of $S_n$ on $\Phi_n$.
It is obvious that the maps $\theta_{E}$ and $\theta_{St}$ are $S_n$ equivariant. Let $\delta$ 
denote the nilpotent set $$\{ (i, j)~|~i < j, 1 \leq i, j \leq n \} \subset \Phi_n.$$ The set $\delta$ is a maximal nilpotent 
subset. It is easy to check that any nilpotent subset is contained in a maximal nilpotent subset and that any maximal 
nilpotent subset is conjugate under the action of $S_n$ to $\delta$. 

To prove that $\theta$ defines a natural isomorphism of $\theta_{E}$ and $\theta_{\St}$, we must show that 
for any $\alpha \in \Phi_n$, the surjective canonical homomorphism $\St_n(R)_{\alpha} \rightarrow \E_n(R)_{\alpha}$ is 
injective as well. By the previous paragraph, it suffices to consider the maximal nilpotent set $\delta$. 
But here the result follows immediately from \cite{Mil}, Lemma 9.14. 

It was shown, following the definition of the global action $\GL_n(R)$ that $\theta_{E}$ preserves intersections. 
$\theta_{\St}$ preserves intersections because of the following facts: 
$\theta_{E}$ preserves intersections, each canonical homomorphism $(\St_n(R))_{\alpha} \rightarrow (\E_n(R))_{\alpha}$ 
is bijective and $(\E_n(R))_{\alpha \cap \beta} = (\E_n(R))_{\alpha} \cap (\E_n(R))_{\beta}$. 

\item[{\rm[2.]}] The left hand isomorphism follows immediately from $[1.]$ above. The right hand isomorphism is defined and 
is obviously surjective. Using the definition of the Steinberg group by generators and relations, one can construct
straightforward an inverse to this homomorphism, since any relation is contained in some local 
subgroup $(\St_n(R))_{\alpha}$. 

\item[{\rm[3.]}] Let $x \in (\St_n(R))_{\alpha}$ and $y \in (\St_n(R))_{\beta}$. Let $\gamma = \alpha \cap \beta$.
Suppose $\theta(x) = \theta (y)$. We must show $x = y$. Clearly, $\theta(x) = \theta(y)$ in $(\E_n(R))_{\gamma}$. 
Let $z \in (\St_n(R))_{\gamma}$ be such that $\theta(z) = \theta(x)$. Since $(\St_n(R))_{\gamma} \subset 
(\St_n(R))_{\alpha},$ it follows that $x = z$, since the homomorphism $(\St_n(R))_{\alpha} \rightarrow (\E_n(R))_{\alpha}$ 
is bijective. Similarly, $y = z$. 
\end{itemize} 
\end{proof}

\section{Elementary homotopy theory of global actions} 
                       
This section summarizes in a convenient form the constructions and results we need from the homotopy theory of global actions, in particular of single domain global actions. They are due to the first named author.

\subsection{The notion of homotopy} 

The most natural notion of homotopy is the following. 

   To begin we recall the notion of product for global actions. Suppose $A$ and $B$ are global actions with underlying sets 
$X$ and $Y$ and indexing sets $\Phi_A$ and $\Phi_B$, respectively. Define the {\it product global action} $A \times B$ as follows. 
Its underlying set is the Cartesian product $X \times Y$ of sets and its index set is also the Cartesian product 
$\Phi_{A} \times \Phi_{B}$ with quasi-ordering defined by 
$(a,b) \leq (a',b')$ if and only if $a \leq a'$ and $b \leq b'$. The local set $(X \times Y)_{(a,b)}$ is the Cartesian product 
$X_{a} \times X_{b}$ and the local group $G_{(a,b)}$ is the product group $G_{a} \times G_b$. Its action on 
$(X \times Y)_{(a,b)}$ is the obvious one, namely coordinatewise. 

Let $f,g: A \rightarrow B$ denote morphisms of global actions. Let $L$ denote the line action, cf. Section 2, with underlying set $\Z$. 
For $n \in \Z$ let $\iota_{n}: X \rightarrow X \times L$, $x \rightarrow (x,n)$. It clearly defines a morphism 
$\iota_n: A \rightarrow A \times L$ of global actions. The morphisms $f$ and $g$ are called {\it homotopic} if there is a morphism 
$H: A \times L \rightarrow B$ of global actions and integers $n_{-} \leq n_{+}$ such that for all $n \leq n_{-}$, 
$fH\iota_n = fH\iota_{n_-}$ and for all $n_{+} \leq n$, $gH\iota_n = gH\iota_{n_+}$. The morphism $H$ is called,
as in topology, a {\it homotopy} from f to g.

In some situations such as that of paths, a variant of the above concept
is needed. We shall call the one needed for paths, stable homotopy, and
define it in the next subsection. 
(In lecture notes distributed in the past, it was called end-point homotopy or end-point stable homotopy.)

\subsection{Stable homotopy of paths and the fundamental group} 

The goal of this section is to define the notion of stable homotopy for paths and to define the fundamental group
functor $\pi_1$. In passing we define the path connected component functor $\pi_0$.

   Throughout this section $A$ and $B$ denote global actions with underlying sets $X$ and $Y$, respectively, and $L$ the line action.

   The easiest and most natural way to define a {\it path} in $A$ is as a finite sequence 
$x_1, \ldots ,x_n$ of points $x_i \in X$ such that for each $i < n$ there is an element $g$ in some local group of $A$ such that 
$x_i$ is in the domain of the action of $g$ and $x_i g = x_{i+1}$. The following equivalent definition is better for the 
stable homotopy theory we need and shall develop.

\begin{defn}\label{defn:path}
Let $\omega : L \rightarrow A$ denote a morphism. We say that
it is {\it stable on the left} or simply {\it left stable} if there is an integer $n_{-}$ such that for all $n \leq n_{-}$, 
$\omega(n) = \omega(n_{-})$. In this case we say that $\omega$ {\it stabilizes on the left} to $x = \omega(n_-)$. 
Similarly we say that $\omega$ is {\it stable on the right} or simply {\it right stable} if there is an integer $n_{+}$ such that for 
all $n \geq n_+$, $\omega(n) = \omega(n_+)$. In this case we say that $\omega$ {\it stabilizes on the right} to $x = \omega(n_+)$. 
We say that $\omega$ is {\it left-right stable} if it is stable both on the left and on the right. In this case we can 
clearly assume that $n_- \leq n_+$.  A {\it path} is a left-right stable morphism $\omega: L \rightarrow A$. 
A {\it loop} is a path which stabilizes on the left and on the right to the same element of $X$.

   A path $\omega: L \rightarrow A$ is {\it constant} if $\omega(n) = x$ for all $n \in \Z$ and some fixed $x \in X$. 
If $\omega$ is not constant then it is always the case that $n_- < n_+$. On the other hand, if $\omega$ is constant then 
$n_{-}$ and $n_{+}$ can be any integers. For this reason, we exclude constant paths from the following definition.
\end{defn}

\begin{defn}\label{defn:upperlower}
Let $\omega$ denote a nonconstant path. The {\it lower} or {\it left degree} of $\omega$ is defined by

       $${\rm ld}(\omega) = {\rm sup}\{ n_{-} \in \Z~|~ \omega(n) = \omega(n_-) {\rm ~for~all~} n \leq n_- \}.$$

The {\it upper} or {\it right degree} of $\omega$ is defined by

        $${\rm ud}(\omega) = {\rm inf}\{n_{+} \in \Z~|~\omega(n) = \omega(n_+) {\rm ~for~all~} n \geq n_+ \}.$$ 
\end{defn}

   Next we define the notion of composition for paths.

\begin{defn}\label{defn:composition} 
Let $\omega$ and $\omega'$ denote two paths. The {\it initial point} 
(in $X$) of a nonconstant path $\omega$  is defined by 

                       $${\rm in}(\omega) = \omega({\rm ld}(\omega)).$$

The {\it terminal point} of a nonconstant path $\omega$ is defined by

                       $${\rm ter}(\omega) = \omega({\rm ud}(\omega)).$$ 

The {\it initial} and {\it terminal} points of a constant path $\omega$
taking the constant value $x \in X$ is defined by

                     $${\rm in}(\omega) = {\rm ter}(\omega) = x.$$ 

The {\it composition} $\omega \cdot \omega'$ of paths $\omega$ and $\omega'$ exists if ${\rm ter}(\omega) = {\rm in}(\omega')$ and
is defined as follows:

\begin{center}
$(\omega \cdot \omega')$ $= \left\{ \begin{array}{lll}
\omega & \mbox{if $\omega'$ is constant } \\ 
\medskip
\omega'            & 
\mbox{if $\omega$  is constant.} \end{array} \right.$
\end{center}

If $\omega$ and $\omega'$ are nonconstant then  
\begin{center}
$(\omega \cdot \omega')(n)$ $= \left\{ \begin{array}{lll}
\omega'(n) & \mbox{for all $n \leq {\rm ud}~\omega'$,} \\ 
\medskip
\omega(n - {\rm ud}~\omega' + {\rm ld}~\omega)            & 
\mbox{for all $n \geq {\rm ud}~\omega'$.} \end{array} \right.$
\end{center} 
          
It is clear that the composition law $\cdot$ on paths is associative.
\end{defn} 
 
    We turn now to the notion of stable homotopy for paths.

\begin{defn}
A homotopy $H: L \times L \rightarrow A$ is called a {\it stable}  (or {\it end-point stable}) homotopy of paths if for any 
$n \in \Z$, $H\iota_n$ is a path and if for any pair $m,n \in \Z$,  
${\rm in}(Hi_m) = {\rm in}(Hi_n)$ and ${\rm ter}(H \iota_m) = {\rm ter}(H \iota_n)$. 

Suppose there is a homotopy $H: L \times L \rightarrow A$ and there exist integers $n_{-}, n_{+}$ such that 
$\omega$ is the unique path  with the property $\omega = H\iota_{n}$ for all $n \leq n_{-}$ and if $\omega'$ is the unique path such that 
there is an integer $n_{+}$ with the property $\omega' = H\iota_{n}$ for all $n \geq n_{+}$ then we say that $\omega$ is
{\it stably homotopic} to $\omega'$ and write $\omega \simeq \omega'$.
\end{defn}

   The notion of homotopy is a generalization of the notion of path and has a notion of composition such that the compostion of 
two end-point stable homotopies of paths is again an end-point stable homotopy of paths. We shall take the time to explain 
this systematically, by replacing the line action $L$ in Definitions \ref{defn:path}, \ref{defn:upperlower} and 
\ref{defn:composition} above, by any 
global action $B \times L$ where $B$ is an arbitrary global action. In other words, we are replacing the trivial global action, 
consisting of one point being acted on by the trivial group, 
by an arbitrary global action $B$. Thus instead of moving a point through the space $A$, as in the case of a path, we are moving 
one space $B$ through another space $A$. In this setting the notion of {\it constant homotopy} becomes a morphism 
$H: B \times L \rightarrow A$ such that for any pair $m,n \in \Z$ $H\iota_m = H\iota_n$.

   We give now the analogues of Definitions \ref{defn:path}, \ref{defn:upperlower} and \ref{defn:composition}.

\begin{defn}
Let $H: B \times L \rightarrow A$ denote a morphism. We say that
it is {\it negatively stable} (or lower stable) if there is an integer $n_{-}$ such that for all $n \leq n_{-}$, 
$H\iota_n = H\iota_{n_{-}}$. In this case we say that $\omega$ {\it stabilizes negatively} (or below) to $f = H\iota_{n_{-}}$. 
Similarly we say that $\omega$ is {\it positively stable} (or upper stable) (italex) if there is an integer $n_{+}$ 
such that for all $n \geq n_{+}$, $H\iota_n = H\iota_{n_{+}}$. In this case we say that $\omega$ {\it stabilizes positively} 
(or above) to $g = H\iota_{n_{+}}$. We say that $\omega$ is a {\it homotopy} if it is both negatively and positively stable. 
In this case we say that $f$ is {\it homotopic} to g. Clearly this definition of homotopy for 
morphisms $B \rightarrow A$ is identical with that in Section \S $3.1$. 
\end{defn}

\begin{defn}
Let $H$ denote a homotopy. The {\it  negative or lower degree} of $H$ is defined by

        $${\rm ld}(H) = {\rm sup}\{n_{-} \in \Z~|~H\iota_{n} = H\iota_{n_{-}} {\rm for~all~} n \leq n_{-} \}.$$

The {\it positive} or {\it upper degree} of H is defined by

       $${\rm ud}(H) = {\rm inf}\{ n_{+} \in \Z~|~H\iota_{n} = H\iota_{n_{+}} {\rm for~all~} n \geq n_{+} \}.$$ 

\end{defn} 

   Next we define the notion of composition for homotopies.

\begin{defn}
Let $H$ and $H'$ denote homotopies $B \times L \rightarrow A$. The {\it initial morphism} in  $Mor(B,A)$ of a nonconstant homotopy  
is defined by 

                       $${\rm in}(H) = Hi_{{\rm ld}(H)}.$$

The {\it terminal morphism} in $Mor(B,A)$ of a nonconstant $H$ is defined by

                       $${\rm ter}(H) = Hi_{{\rm ud}(H)}.$$ 

The {\it initial} and {\it terminal} morphism of a constant homotopy
$H$ taking the constant value $f$ in $Mor(B, A)$ is defined by

                     $${\rm in}(H) = {\rm ter}(H) = f.$$

If $H$ is a stable homotopy of paths then ${\rm in}(H)$ is called the {\it initial path} and ${\rm ter}(H)$ the {\it terminal path}.  
The {\it composition} $H \cdot H'$ of homotopies $H$ and $H'$ exists if ${\rm ter}(H) = {\rm in}(H')$ and is defined as follows:

\begin{center}
$(H \cdot H')$ $= \left\{ \begin{array}{lll}
H & \mbox{if $H'$ is constant } \\ 
\medskip
H'            & 
\mbox{if $H$  is constant.} \end{array} \right.$
\end{center}

If $H$ and $H'$ are nonconstant then

\begin{center}
$(H \cdot H')(n)$ $= \left\{ \begin{array}{lll}
H'(n) & \mbox{for all $n \leq {\rm ud}~H'$,} \\ 
\medskip
H(n - {\rm ud}~H' + {\rm ld}~H)            & 
\mbox{for all $n \geq {\rm ud}~H'$.} \end{array} \right.$
\end{center}          

Clearly if $f' = {\rm in}(H')$ and $f = {\rm ter}(H)$ then ${\rm in}(H.H') = f'$ and ${\rm ter}(H.H') = f$. 
Thus the relation of homotopy on morphisms in $Mor(B,A)$ is an equivalence relation. Furthermore if 
$H$ and $H'$ are composable and are at the same time stable homotopies of paths then the composition $H \cdot H'$ is also a 
stable homotopy of paths. This shows that the relation of stable homotopy on paths is an equivalence relation. It is clear that  
composition law $\cdot$ is associative, although we won't need this fact.
\end{defn}

    There is another important way to compose stable homotopies, but not
arbitrary homotopies, which goes as follows. This kind of composition
will be denoted by a square ${\square}$.

\begin{defn}\label{defn:pathhomotopy} 
Let $H: L \times L \rightarrow A$ be a stable homotopy of paths. By definition the elements 
${\rm in}(Hi_n)$ and ${\rm ter}(Hi_n)$ do not depend on the choice of $n$. Define the {\it initial point} (as opposed to intial
path) of $H$ by

                 $${\rm inp}(H) = {\rm in}(Hi_n) {\rm~for~ any~ n \in \Z}.$$ 

Define the {\it terminal point} (as opposed to terminal path) of $H$
by

                 $${\rm terp}(H) = {\rm ter}(Hi_n) {\rm~for~any~ n \in \Z}.$$ 

If $H$ and $H': L \times L \rightarrow A$ are stable homotopies of paths such that ${\rm terp}(H) = {\rm inp}(H')$ then the 
composition $H \square H'$ is defined and has the property that
${\rm in}(H \square H') = {\rm in}(H) \cdot {\rm in}(H')$ and ${\rm ter}(H \square H') = {\rm ter}(H) \cdot {\rm ter}(H')$. 
Moreover the composition law $\square$ is associative, although we won't need this fact.   

   We leave the construction of $\square$ to the interested reader. Theorem \ref{thm:onestep} 
below says that elementary stable homotopies are the only tools one needs
to construct $H \square H'$.
\end{defn}

   We are now in a position to construct the fundamental monoid $\Pi_1(A)$
and fundamental group $\pi_1(A)$ of a pointed global action $A$.

\begin{defn}\label{defn:fundamentalmonoid} 
Let $A$ denote a pointed global action with base point $\circ$ in $X$. The {\it fundamental monoid} 

                   $$\Pi_1(A) = \Pi_1(A, \circ)$$

is the set of all loops at $\circ$, with composition given by the composition law of Definition \ref{defn:composition}  
and identity the constant loop at $\circ$.
\end{defn}

    We want to construct the fundamental group $\pi_1$ from $\Pi_1$. For this we need a definition and a result.

\begin{defn}
A {\it 1-step stable homotopy} $H: L \times L \rightarrow A$ is either a constant stable homotopy or a nonconstant homotopy such that 
${\rm ud}(H) - {\rm ld}(H) = 1$. Clearly every stable homotopy of paths is a composition of
a finite number of $1$-step stable homotopies. Let $n = {\rm ld}(H), \omega = Hi_n,$ and $\omega' = H{i_{n+1}}$. 
A $1$-step homotopy $H$ is called {\it elementary},
if             \begin{center} 
               \begin{itemize} 
               \item[(3.9.1)] it is constant,
               \end{itemize}
               \end{center}
or the following holds. 
There is an $i \in \Z$ such that for all $j \leq i$ and all $j \geq i+2$, 
$\omega(j) = \omega'(j)$ and there are elements $x,y \in X$ satisfying one of the following:

\begin{itemize}

\item[(3.9.2)] 
               $
               \begin{pmatrix}              
               w'(i) & w'(i+1) & w'(i+2)\\  
               w(i) & w(i+1)  & w(i+2)      
               \end{pmatrix}$                
               $= \begin{pmatrix}
 x,y,y \\
 x,x,y
\end{pmatrix}$ 
               
\item[(3.9.3)] $
               \begin{pmatrix}               
               w'(i) & w'(i+1) & w'(i+2) \\       
               w(i) & w(i+1) & w(i+2)         
               \end{pmatrix} 
               $
$
= \begin{pmatrix}
   x,x,y \\ 
  x,y,y
\end{pmatrix} 
$

\item[(3.9.4)] $
               \begin{pmatrix}               
               w'(i) & w'(i+1) & w'(i+2) \\     
               w(i)  & w(i+1)  & w(i+2)          
               \end{pmatrix}
               $
$
= \begin{pmatrix}
x,y,x \\
x,x,x
\end{pmatrix} 
$

\item[(3.9.5)] $
               \begin{pmatrix}               
               w'(i) & w'(i+1) & w'(i+2) \\         
               w(i) & w(i+1) & w(i+2)         
               \end{pmatrix}                   
               $
$= \begin{pmatrix} 
x,x,x \\
x,y,x
\end{pmatrix}.$ 
\end{itemize} 
               
\end{defn}

\begin{thm}\label{thm:onestep} 
Every $1$-step stable homotopy is a composition of elementary homotopies and thus every stable homotopy is a composition of elementary
homotopies.
\end{thm} 

    The proof is not very difficult and is left to the reader. However,
in the current paper, we do not use the fact that elementary homotopies generate all stable homotopies, rather we use them, 
as in the proof of Corollary \ref{cor:omegaomegainverse} below, to show directly that certain homotopies exist.

\begin{defn}
 If $\omega$ is a path, define the {\it inverse path} $\omega^{-1}$ by

                      $$\omega^{-1}(n) = \omega(-n).$$ 
\end{defn}

\begin{cor}\label{cor:omegaomegainverse} 
If $\omega$ is a path then $\omega \cdot \omega^{-1}$ is stably homotopic to the
constant path at ${\rm in}(\omega)$.
\end{cor} 

   The Corollary follows by an easy application of elementary homotopies.

\begin{defn}
Let $A$ denote a pointed global action. By \ref{defn:fundamentalmonoid} stable
homotopy respects composition in $\Pi_1(A)$. Thus the stable homotopy classes of loops in $\Pi_1(A)$ form a monoid with 
identity the stable homotopy class of the constant loop at the base point. By Corollary \ref{cor:omegaomegainverse} every
loop $\omega \in \Pi_1(A)$ has up to stable homotopy an inverse $\omega^{-1}$. Thus
the stable homotopy classes of loops in $\Pi_1(A)$ form a group which we denote by
                                $$\pi_1(A)$$ 
and call the {\it (algebraic) fundamental group} of $A$.

   Two points $x,x' \in X$ are called {\it path connected} if there is
a path $\omega$ such that ${\rm in}(\omega) = x$ and ${\rm ter}(\omega) = x'$. The composition law for paths shows 
that the relation path connected is transitive, the construction of the inverse path $\omega^{-1}$ shows that the relation is
symmetric, and the existence of the constant path at any point shows that
the relation is reflexive. Thus the relation path connected is an
equivalence relation on X.
\end{defn}

\begin{defn}
Let $$\pi_0(A)$$ 
denote the equivalence classes of the relation path connected on $X$.
It is called the set of {\it path connected} components of $A$. If
$A$ has a base point then $\pi_0(A)$ is usually given as base point, the equivalence class of the base point of $X$.
\end{defn}

\subsection{Path connected component of the unimodular row global action} 
We now decribe the unimodular row global action and compute its path connected component. 

{\bf The unimodular global action:}  The {\it unimodular global action} has as underlying set $\Um_n(R)$, the set of all 
$R$-unimodular row vectors $v = (v_1,v_2, \ldots,v_n)$ of length $n$, with coefficients $v_i \in R$. 
Recall that unimodular means there is a row vector $w = (w_1, \ldots ,w_n)$ such that $v(^tw) = \sum_{i} v_iw_i = 1$, 
where $t$ denotes the transpose operator on (not necessarily square) matrices.
(The row $w$ is automatically unimodular, because $1 = ^t1 = ^t(v(^tw)) = w(^tv).$) 
The general linear group $\GL_n(R)$ acts on $\Um_n(R)$ on the right, in the usual way. 
The indexing set $\Phi_n$  as well as the local groups $(\E_n(R))_{\alpha}$ are the same as for the global action
$\GL_n(R)$. Each local set is the whole $\Um_n(R)$ and the action of each local group $\E_n(R)$ on $\Um_n(R)$ is via that of 
$\GL_n(R)$ on $\Um_n(R)$. Abusing notation, we shall let $\Um_n(R)$ denote also this (single domain) global action. 
We give the underlying set of $\Um_n(R)$ the distinguished point $e = (1,0, \ldots, 0)$.

\begin{prop}
$\pi_0(\Um_n(R)) = \Um_n(R)/\E_n(R)$. Give this coset space the base point $e\E_n(R).$ Then,
the connected component of $e$ in $\Um_n(R)$ is the coset space $e\E_n(R).$ 
\end{prop}
\begin{proof}
We prove that $v, w$ belong to the same path component in $\Um_n(R)$ if and only if there exists 
$\varepsilon \in \E_n(R)$ such that $v\varepsilon = w$ i.e., if and only if $v\E_n(R) = w\E_n(R).$ 

Let $v, w \in \Um_n(R)$ belong to the same path component and let $\omega$ be a path from $v$ to $w.$ 
As $\omega$ is a morphism of global actions, there exist $\varepsilon_i \in {\GL_n(R)}_{\alpha_i},$
$1 \leq i \leq N$ such that $v\varepsilon_{1}\cdots\varepsilon_{N} = w.$ Then $\varepsilon:= \prod_i \varepsilon_i \in \E_n(R)$
has the required property. 

Conversely suppose $w = v\varepsilon,$ for some $\varepsilon \in \E_n(R).$ Hence there exist $\E_{ij}(\lambda),$
$\lambda \in R$ such that $\varepsilon = \prod \E_{ij}(\lambda).$ As each $\E_{ij}(\lambda)$ lies in some local 
set, we can easily define a path from $v$ to $w.$ Thus, $\pi_0(\Um_n(R)) = \Um_n(R)/\E_n(R)$. 

From this also follows that the path component of the base point $e$ is $e\E_n(R).$  
\end{proof}

We now introduce a global actions structure on the coset space above. We will introduce another important global action 
which is a certain coset space of the Steinberg group. These global actions are crucial 
in computing the fundamental group of the unimodular row global action. 

\begin{itemize} 

\item {\bf The elementary unimodular global action:} The {\it elementary unimodular global action} 
has as underlying set $\EUm_n(R) = eE_n(R),$ {\it the path connected component of the base point in $\Um_n(R).$ } 
The index set as well as the local groups are the same as
those for $\Um_n(R)$. Each local set is the whole of $\EUm_n(R)$. Abusing notation, as usual, we let $\EUm_n(R)$ denote 
this global action. We give it the base point $e$.

\item {\bf The Steinberg unimodular global action:} 
Let $\P_n(R)$ denote the subgroup of $\GL_n(R)$ which leaves $e$ fixed. 
Clearly each matrix in $\P_n(R)$ takes the form $\begin{pmatrix} 
1 & 0 \\ 
v & \tau 
\end{pmatrix},$ for some $v \in M_{n-1, 1}(R),$ $\tau \in \GL_{n-1}(R).$ 

Let $\EP_n(R) = \P_n(R) \cap \E_n(R)$. There is an obvious canonical identification $$\EUm_n(R) = \E_n(R)/\EP_n(R)$$
of global actions, which is induced by sending each element $e\varepsilon$ in $\EUm_n(R)$ to $\varepsilon\EP_n(R).$ 
The global action on the right is the obvious one: the underlying set is the right coset space $\E_n(R)/\EP_n(R)$ and 
the index set $\Phi_n$ and local groups $\E_n(R)_{\alpha}$ are the same as for $\Um_n(R)$. The local sets are
all of $\E_n(R)/\EP_n(R)$ and the action of each local group on $\E_n(R)/\EP_n(R)$ is induced by the natural 
right action of the group $\E_n(R)$ on it. Let $\theta : \St_n(R) \rightarrow \E_n(R)$ denote the canonical homomorphism. 
Let 
$$\B_n(R) = \langle x^{-1}abx \in \theta^{-1}(\EP_n(R))~|~x \in \St_n(R), a \in \St_n(R)_{\alpha}, b \in St_n(R)_{\beta},$$
$${\rm for~some~\alpha, \beta~in~\Phi_n}\rangle.$$ 

Clearly $\B_n(R)$ is a normal subgroup of $\theta^{-1}(\EP_n(R))$. 

The {\it Steinberg unimodular global action} $\StUm_n(R)$ has underlying set the right coset space 
$\St_n(R)/\B_n(R)$. The indexing set $\Phi_n$ and local groups $\St_n(R)_{\alpha}$ are the same as those of
the Steinberg global action $\St_n(R)$. Each local set is all of $\St_n(R)/\B_n(R)$ and the action of each 
$\St_n(R)_{\alpha}$ on $\St_n(R)/\B_n(R)$ is induced by the natural right action of the group $\St_n(R)$ on it.
Abusing notation, we shall denote the Steinberg unimodular action also by $\St_n(R)/\B_n(R)$. 
We give it the distinguished point $e\B_n(R)$. It is easy to check that this is a path-connected 
global action. 

     There is a canonical base point preserving morphism
$$\St_n(R)/\B_n(R) \rightarrow \E_n(R)/\EP_n(R)$$ of coset spaces and global actions, which is induced by the map

     $$\St_n(R) \rightarrow \EUm_n(R)$$ sending each $x$ in $\St_n(R)$ to $e\theta(x)$.
\end{itemize} 

We recall the definition of a covering morphism (see \cite{Bak}) by introducing another important global action: the 
star global action. 

\begin{defn}\label{defn:star}
Let $A$ be a path-connected global action with underlying set $X$, index set $\Phi$ and local groups 
$X_{\alpha} \curvearrowleft  G_{\alpha}$, for $\alpha \in \Phi$. 
Given $\alpha \in \Phi$ and $x \in X$, 
let 
${\rm{star}}(x)$ denote the following global action: 
the underlying set $X_{{\rm{star}}(x)}$ is 
the union of all $x G_{\alpha}$ where $G_{\alpha}$ is a local group which acts on $x$ i.e., 
$$X_{{\rm{star}}(x)} = \underset{\{\alpha \in \Phi| x \in X_{\alpha} \}}{\cup} x \cdot G_{\alpha}$$
The index set $\Phi_{{\rm{star}}(x)}$ consists of all $\alpha \in \Phi$ such that $G_{\alpha}$ acts on $x$. 
$\Phi_{{\rm{star}}(x)}$ inherits its ordering from $\Phi$. 
If $\alpha \in  \Phi$, then $(X_{{\rm star}(x)})_{\alpha} = xG_{\alpha}$ and  $(G_{{\rm star}(x)})_{\alpha} = G_{\alpha}$. 
\end{defn}

\begin{defn}\label{defn:covering} 
A morphism $f: B \rightarrow A$ of global actions is {\it surjective}, if it is surjective map on the underlying sets.  
A surjective morphism $f: B \rightarrow A$ of global actions is called a {\it covering morphism} if for every 
$b \in X_{B}$, the induced map $f:{\rm star}(b) \rightarrow {\rm star}(f(b))$ 
is an isomorphism of global actions.
\end{defn} 

   The next proposition records some facts 
which are needed for the (algebraic) homotopy theory of $\Um_n(R)$.

\begin{cor}\label{cor:coveringsteinberg} 
The canonical homomorphism $\St_n(R) \rightarrow \E_n(R)$ of path-connected 
global actions is a covering morphism in the sense of \cite{Bak}, \S 10.  
\end{cor}

\begin{cor} 
The canonical morphism $\StUm_n(R) \rightarrow \EUm_n(R)$ is
a covering morphism of path-connected global actions.
\end{cor}

\section{Coverings, fundamental group and elementary unimodular row global action} 
In this section we state (without proof) results for homotopy theory in the framework of global actions. 
The interested reader should refer to \cite{Bak}, \S 11. 
These will be applied in the concrete case of the elementary unimodular row global action to 
compute its fundamental group explicitly. 

We give some basic definitions and then outline the construction of a connected, simply connected covering of the 
path-connected single domain global action $(\EUm_n(R), e\EP_n(R))$. The checking of details 
is not difficult and is left to the reader. 

\begin{defn}\label{defn:simplyconnected} 
A path-connected global action $A$ with base-point $a_0$ is said to be simply connected, if 
the fundamental group of $A$ at $a_0$ is trivial i.e., $\pi_1(A, a_0) = e.$ 
\end{defn}

\begin{thm}\label{thm:model-special-case}
There exists a path-connected, simply connected base point preserving covering of the path-connected global action 
$(\EUm_n(R), e\EP_n(R)).$ Moreover, it is of the form $\E_n(R)/H_{B},$ where $H_{B}$ is a normal subgroup of $\EP_n(R).$
\end{thm}
\begin{proof}
Follows from ``Structure theorem'' for single domain global actions 
(see \cite{Bak}, Definition $3.3$, Theorem $11.1$, Proposition $11.3$.) 
\end{proof} 

We now would like to prove that the path-connected, simply connected covering of $(\EUm_n(R), e\EP_n(R))$ is also universal.
For this, we state without proof the ``Lifting criterion'' in the context of global actions. 

\begin{lem}\label{lem:lift}
Let $q: (B, b_0) \rightarrow (\EUm_n(R), e\EP_n(R))$ be a pointed covering morphism of path-connected global actions 
Let $C$ be a path connected global action with base point $c_0.$ 
Let $f: (C, c_0) \rightarrow (\EUm_n(R), e\EP_n(R))$ be a pointed morphism. 
Then, a morphism $\widetilde{f}: (C, c_0) \rightarrow 
(B, b_0)$ lifting $f$ exists if and only if $f_{*}(\pi_1(C, c_0)) \subset q_{*}(\pi_1(B, b_0))$. 
Moreover, if  $f$ exists, then it is unique. 
\end{lem}

\begin{cor}\label{cor:uniqueuniversal} 
Every path-connected, simply-connected covering of the connected single domain global action $(\EUm_n(R), e\EP_n(R))$
is universal. In particular, any two path-connected, simple connected coverings of the global action $(\EUm_n(R), e\EP_n(R))$
are isomorphic.   
\end{cor}
\begin{proof}
Let $f: (C, c_0) \rightarrow (\EUm_n(R), e\EP_n(R))$ be a morphism to 
the global action $(\EUm_n(R), e\EP_n(R))$ from a 
path-connected, simply connected covering of $(\EUm_n(R), e\EP_n(R))$. 
Then the lifting criterion ensures that $(C, c_0)$ is universal, as $f_{*}(\pi_1(C, c_0)) = e.$ Universality 
then implies that any two path-connected, simple connected coverings of $(\EUm_n(R), e\EP_n(R))$
are isomorphic.    
\end{proof}

\begin{cor}\label{cor:Fiber} 
Let $(\E_n(R)/H_{B}, eH_{B})$ be the path-connected, simple connected of $(\E_n(R)/\EP_n(R), e\EP_n(R))$ with pointed 
covering morphism $p: (\E_n(R)/H_{B}, eH_{B}) \rightarrow (\E_n(R)/\EP_n(R), e\EP_n(R)).$ 
Then, $$\pi_1(\EUm_n(R), e\EP_n(R)) \simeq p^{-1}(e\EP_n(R)).$$ 
\end{cor} 
\begin{proof}
See \cite{Bak}, Theorem $10.17.$ 
\end{proof}

We can view the fault line between algebra and topology more clearly. 
This helps us to compute the fundamental group of $\EUm_n(R)$ explicitly. 
The notion of universal cover is clear: Given $\EUm_n(R)$, a universal cover 
$X \rightarrow \EUm_n(R)$ is a cover such that given any other cover $X' \rightarrow \EUm_n(R)$ there is unique base point 
preserving map $X \rightarrow X'$ making the usual diagram commute. The existence of a universal cover for a single domain 
action has two distinct proofs, one algebraic and the other topological. 
The algebraic proof shows that $\StUm_n(R) \rightarrow \EUm_n(R)$ is a universal cover. 
The topological proof seen above, shows that any connected simply connected cover $Y^{~}$ is universal and explicitly 
constructs one. The universality of the topological and algebraic constructions yields a unique isomorphism
$\StUm_n(R) \rightarrow Y^{~}$ making the usual diagram commute. 

We now record the key observation regarding coverings in coset spaces. This leads us to explicitly computing 
the fundamental group of $\EUm_n(R).$

\begin{prop}\label{prop:two} 
Let $n \geq 3$ be an integer and let $K \subset H$ be subgroups of $\E_n(R)$. 
Then, $p: \E_n(R)/K \rightarrow \E_n(R)/H$ defined by $p(K\varepsilon) = H\varepsilon$ is a 
covering morphism of global actions if and only if $H_2 \subset K$, where 
$H_2 = \langle H \cap x^{-1}(\E_n(R))_{\alpha}(\E_n(R))_{\beta}x \rangle$ 
i.e., $H_2$ is the subgroup of $H$ generated by all elements in $H$ which are also of the form  
$x^{-1}\varepsilon_{\alpha}\varepsilon_{\beta}x$, for some $x \in \E_n(R)$ and some local group elements 
$\varepsilon_{\alpha}, \varepsilon_{\beta}$. 
\end{prop}
\begin{proof}
It is easy to check that $p$ is a surjective morphism of global actions. We now prove that $p$ is injective on stars if 
$H_2 \subset K$ i.e., we prove that $p: {\rm star}~(K\varepsilon) \rightarrow {\rm star}~(H\varepsilon)$ is injective 
if $H_2 \subset K$. Let $K\varepsilon_1, K\varepsilon_2 \in {\rm star}~(K\varepsilon)$ with $p(K\varepsilon_1) = p(K\varepsilon_2)$ i.e., 
$\varepsilon_2\varepsilon_1^{-1} \in H$. 

As $K\varepsilon_1, K\varepsilon_2 \in {\rm star}~(K\varepsilon)$, there exist local group elements 
$\varepsilon_{\alpha}, \varepsilon_{\beta} \in (\E_n(R))_{\alpha}, (\E_n(R))_{\beta}$ respectively such that 
$K\varepsilon_1 = (K\varepsilon)\varepsilon_{\alpha}$ and 
$K\varepsilon_2 = (K\varepsilon)\varepsilon_{\beta}$. Now $p(K\varepsilon_1) = p(K\varepsilon_2)$ implies 
$\varepsilon \varepsilon_{\alpha}\varepsilon_{\beta}^{-1} \varepsilon^{-1} \in H$ i.e., 
$\varepsilon \varepsilon_{\alpha}\varepsilon_{\beta}^{-1} \varepsilon^{-1} \in H_2$. Thus 
$\varepsilon \varepsilon_{\alpha}\varepsilon_{\beta}^{-1}\varepsilon^{-1}  \in K$, as $H_2 \subset K$. 
This proves $\varepsilon_2(\varepsilon_1)^{-1} \in K$ i.e., $p$ is injective, if $H_2 \subset K.$ 

Conversely suppose that $\E_n(R)/K \rightarrow \E_n(R)/H$ is a covering morphism. 
It is enough to prove that every generator of $H_2$ lies in $K$, as $K$ is a subgroup. 
Let $x \varepsilon_{\alpha} \varepsilon_{\beta}x^{-1} \in H.$ Then $(Kx)\varepsilon^{-1}_{\alpha}, (Kx)\varepsilon_{\beta} \in 
{\rm star}(Kx),$ with $p((Kx)\varepsilon^{-1}_{\alpha}) = p((Kx)\varepsilon_{\beta}).$  Injectivity of $p$ on ${\rm star}(Kx)$ 
implies $Kx\varepsilon^{-1}_{\alpha} = Kx\varepsilon_{\beta},$ i.e., $x \varepsilon_{\alpha} \varepsilon_{\beta}x^{-1} \in K,$ 
proving that every generator of $H_2$ lies in $K.$ 
\end{proof}

\begin{thm}\label{thm:IMP1} 
Let $\EUm_n(R)$ be the connected single domain global action with base point $e\EP_n(R)$.
Then, the simply connected (universal) covering of $\EUm_n(R)$ is $\E_n(R)/(\EP_n(R))_2.$ 
Thus $$\pi_1(\EUm_n(R), e\EP_n(R)) \simeq \EP_n(R)/ (\EP_n(R))_2$$ 

$$\simeq \widetilde{P_n}(R)/B_n(R) \simeq \widetilde{P_n}(R)/(\widetilde{P_n}(R))_2$$  
under the isomorphism induced by the homomorphism $\theta: \St_n(R) \rightarrow \E_n(R).$ 
(see Proposition \ref{prop:intersection}). 
Here $(\EP_n(R))_2$ is defined as in Proposition \ref{prop:two} above and the same Proposition also 
shows that $(\widetilde{P_n}(R))_2 \simeq B_n(R).$ 
\end{thm} 
\begin{proof}
Let the simply connected covering of $\EUm_n(R)$ be given by 
$\E_n(R)/H_{B},$ where $H_{B}$ is a normal subgroup of $\EP_n(R).$ 
Using Proposition \ref{prop:two}, it is easy to check that $\E_n(R)/(\EP_n(R))_2$ is a covering of 
$\E_n(R)/H_{B},$ by observing that $(H_{B})_2 \subset (\EP_n(R))_2$.   

We now prove that $\E_n(R)/(\EP_n(R))_2$ is a simply connected covering of $\E_n(R)/H_{B}.$ 
For this note that $p_{*}(\pi_1(\E_n(R)/(\EP_n(R))_2)) \subset \pi_1(\E_n(R)/H_{B})$, which is trivial as 
$\E_n(R)/H_{B}$ is a simply connected covering of $\EUm_n(R)$. Now as $p_{*}$ is injective,
one has $\pi_1(\E_n(R)/(\EP_n(R))_2)$ is trivial, thus showing that $\E_n(R)/(\EP_n(R))_2$ 
is another simply connected covering of the global action $\EUm_n(R)$. By Corollary \ref{cor:uniqueuniversal} 
we have $\E_n(R)/(\EP_n(R))_2 \simeq \E_n(R)/H_B$ is the universal, path-connected, simply-connected covering of the single 
domain global action $\EUm_n(R)$. Using Corollary \ref{cor:Fiber}, we see that 
$$\pi_1(\EUm_n(R), e\EP_n(R)) \simeq p^{-1}(e\EP_n(R)) = \EP_n(R)/ (\EP_n(R))_2$$ 
$$\simeq \widetilde{P_n}(R)/B_n(R) \simeq  \widetilde{P_n}(R)/(\widetilde{P_n}(R))_2.$$  
\end{proof}

\section{Stability in $K$-theory and fundamental group of unimodular row global action} 

In this section we construct certain exact sequences of pointed sets. Under suitable conditions on the stable rank   
of the ring under consideration these exact sequences of pointed sets turn out to be exact 
sequences of groups. The sandwiching of $\pi_1(\EUm_n(R))$ in this exact sequence of groups helps us to conclude 
(in certain situations) about vanishing of $\pi_1(\EUm_n(R))$. This matches with the well-known algebraic results. 

\subsection{Exact sequences for path-connected and fundamental group functors of unimodular row global action}
In this section we construct exact sequences for path-connected and fundamental group functors of the unimodular 
row global action. Let $\K^s_{0,m}(R)$ be the set of all isomorphism classes of finitely generated projective modules 
$P$ such that for some $r$ (depending on $P$) $P \oplus R^r \simeq R^{m+r}$. 
The base point of $\K^s_{0,m}(R)$ is the isomorphism class of $R^m$. 

\begin{prop}\label{prop:exact-K0-K1} 
Let $R$ be a ring and $n \geq 3$ be an integer. Then, the following exact sequences of pointed sets exist: 
\begin{itemize} 
\item[{1.}] $$\K_{1,n}(R) \xrightarrow{\alpha}  \pi_0(\Um_n(R)) \xrightarrow{\beta} \K^s_{0,n-1}(R) 
\xrightarrow{\gamma} \K^s_{0,n}(R)$$
where the base point of $\K_{1, n}(R)$ is $[I_n]$ and the base point of $\pi_0(\Um_n(R))$ is $[e]$. 

\item[{2.}] $$(\K_{2,n}(R))_2 \xrightarrow{\delta} \K_{2,n}(R) \xrightarrow{\eta} \pi_1(\EUm_n(R)) 
\xrightarrow{\mu} \K_{1,n-1}(R)/(\K_{1,n-1}(R))_2 \xrightarrow{\lambda}$$ 
$$\K_{1,n}(R)$$ 

By definition, $$(\K_{2,n}(R))_2 = \K_{2,n}(R) \cap (\widetilde{P_n}(R))_2;$$  
$$(\K_{1,n-1}(R))_2 = (\GL_{n-1}(R) \cap (\EP_{n}(R))_2)/\E_{n-1}(R),$$
where $(\widetilde{P_n}(R))_2, (\EP_n(R))_2$ are defined 
analogous to $H_2$ in Proposition \ref{prop:two}. 
\end{itemize} 
\end{prop}

\begin{proof} 
\begin{itemize}
\item[{1.}] Define $\alpha: \K_{1,n}(R) \rightarrow \pi_0(\Um_n(R))$ by $\sigma \mapsto (e\sigma)\E_n(R).$ Clearly this map preserves
base points. For defining $\beta: \pi_0(\Um_n(R)) \rightarrow \K^s_{0,n-1}(R)$ note that given $v \in \Um_n(R),$ 
there exists a natural surjective map $\beta_{v}: R^{n} \rightarrow R$ defined by $\beta_{v}(w) = w \cdot v^{t}$ with 
${\rm ker}\beta_{v} \oplus R \simeq R^{n}$ i.e., $[{\rm ker}\beta_{v}] \in \K^s_{0,n-1}(R).$ Define 
$\beta:\pi_0(\Um_n(R)) \rightarrow \K^s_{0,n-1}(R)$ by $\beta([v]) = [{\rm ker}\beta_{v}].$ This is a 
well-defined base-point preserving map. (For details see \cite{SV}.) 
That it is an exact sequence of pointed sets follows using ideas as in Lemma $1.3$ in \cite{SV}. 

\item[{2.}] We first define the maps $\delta, \eta, \mu, \lambda.$ 
\begin{itemize} 
\item[{2(a)}] The map $\delta$ is the natural inclusion map. 

\item[{2(b)}] Define $\eta: K_{2, n}(R) \rightarrow \pi_1(\EUm_n(R))$ by $\eta(Y) = Y(\widetilde{P_n}(R))_2.$ 

\item[{2(c)}] To define the map 
$\mu: \pi_1(\EUm_n(R)) \rightarrow \K_{1,n-1}(R)/(\K_{1,n-1}(R))_2,$ recall that 
$\pi_1(\EUm_n(R)) \simeq \EP_n(R)/(\EP_n(R))_2,$ via the standard homomorphism 
$\theta_n: \St_n(R) \rightarrow \GL_n(R).$
Given $\sigma \in \EP_n(R),$ there exists $\tau \in \GL_{n-1}(R)$ such that 
$\sigma = \begin{pmatrix} 
1 & 0 \\
v & \tau 
\end{pmatrix}
\in \E_n(R).$ Note that $\tau \in \GL_{n-1}(R) \cap \E_n(R)$ defines an element $\tau\E_{n-1}(R)$ of $K_{1, n-1}(R),$ 
which we denote by $[\sigma_{rd}],$ the class of the right diagonal of $\sigma.$ It is easy to check that this 
map is well-defined. 

 Define $\mu(Y(\widetilde{P_n}(R))_2) =[(\theta_n(Y))_{rd}](\K_{1,n-1}(R))_2.$ As 
$\mu(I_n(\widetilde{P_n}(R))_2) =[(\theta(I_n))_{rd}](\K_{1,n-1}(R))_2 = [I_{n-1}](\K_{1,n-1}(R))_2,$ we have  
$\mu(\widetilde{P_n}(R))_2 \subset (\K_{1,n-1}(R))_2.$ 

\item[{2(d)}] The map $\lambda: \K_{1,n-1}(R)/(\K_{1,n-1}(R))_2 \rightarrow \K_{1,n}(R)$ is the natural one induced by the 
right diagonal inclusion of $\GL_{n-1}(R)$ inside $\GL_n(R)$  i.e., the map given by 
$[\tau](\K_{1,n-1}(R))_2 \mapsto \begin{pmatrix} 
1 & 0 \\
0 & \tau 
\end{pmatrix}\E_n(R).$ This map is well-defined: 
for if $[\tau] = [\tau'],$ then $\tau \tau'^{-1} \in (\K_{1,n-1}(R))_2$ i.e., 
$\tau \tau'^{-1} \in \GL_{n-1}(R)$ and 
$\begin{pmatrix} 
1 & 0 \\
0 & \tau \tau'^{-1}
\end{pmatrix} \in (\EP_n(R))_2 \subset \E_n(R).$ Thus, $\begin{pmatrix} 
1 & 0 \\
0 & \tau 
\end{pmatrix}\E_n(R) = \begin{pmatrix} 
1 & 0 \\
0 & \tau' 
\end{pmatrix}\E_n(R).$ 
\end{itemize} 

\item[{3.}] Having defined the maps, we first check that we get a complex. 
\begin{itemize} 
\item[{3(a)}] Note that for $Y \in (\K_{2,n}(R))_2,$
$(\eta \circ \delta)(Y) = \eta(Y) = Y(\widetilde{P_n}(R))_2) = e(\widetilde{P_n}(R))_2),$ as $Y \in (\widetilde{P_n}(R))_2).$ Thus, $\eta \circ \delta = e.$ 

\item[{3(b)}] For $Z \in \K_{2,n}(R)$, consider 
$(\mu \circ \eta)(Z) = \mu(Z(\widetilde{P_n}(R))_2)) = [\theta_n(Z)_{rd}] = [I_n],$ as $Z \in \K_{2,n}(R).$ 
This proves that $\mu \circ \eta = e.$ 

\item[{3(c)}] One also has that $\lambda \circ \mu = e,$ as 
$(\lambda \circ \mu)Z(\widetilde{P_n}(R))_2 = \lambda(\mu(Y(\EP_n(R))_2)) = \lambda([Y_{rd}])$ 
for some $Y \in \EP_n(R)$ via the identification of $\pi_1(\EUm_n(R))$ with the orbit space $\EP_n(R)/(\EP_n(R))_2.$ 
Now $\lambda([Y_{rd}]) = 
[\begin{pmatrix} 
1 & 0 \\
0 & Y_{rd}
\end{pmatrix}]$ $= [I_n] \in \K_{1, n}(R),$ as $Y \in \EP_n(R).$  
\end{itemize} 

\item[{4.}] Now we check exactness at each place: 

\begin{itemize} 

\item[{4(a)}] We check ${\rm ker}~\eta \subset {\rm Im}\delta.$ Let $Y \in \K_{2, n}(R) \in {\rm ker}~\eta.$ 
Hence, $Y \in \widetilde{P_n}(R)_2 \cap \K_{2, n}(R),$ which by definition is $(\K_{2, n}(R))_2.$ Thus, $Y$ belongs 
to ${\rm Im}\delta.$

\item[{4(b)}] We now check that ${\rm ker}~\mu \subset {\rm Im}\eta.$ Let $Y(\widetilde{P_n}(R))_2 \in {\rm ker}~\mu.$
Then, $\mu(Y(\widetilde{P_n}(R))_2) = [(\theta_nY)_{\rm rd}](\K_{1,n-1}(R))_2 = 
[I_n](\K_{1,n-1}(R))_2,$ which implies $(\theta_nY)_{\rm rd} \in (\K_{1,n-1}(R))_2.$ Now we write 
$\begin{pmatrix} 
1 & 0 \\
0 & (\theta_nY)_{\rm rd} 
\end{pmatrix}$
$= \prod_{(i,j)}\varepsilon_{ij}\varepsilon',$ where $\varepsilon_{ij} \in \E_n(R)$ are elementary 
generators and $\varepsilon' \in \E_{n-1}(R).$ Breaking up $\varepsilon'$ further into a product 
of elementary generators, we get that$\begin{pmatrix} 
1 & 0 \\
0 & (\theta_nY)_{\rm rd} 
\end{pmatrix}$
$= \prod_{(i', j')} \varepsilon_{i'j'},$ with $\varepsilon_{i'j'} \in \E_n(R).$  Thus, $X:= \theta_n^{-1}(\begin{pmatrix} 
1 & 0 \\
0 & (\theta_nY)_{\rm rd} 
\end{pmatrix})$ makes sense as an element of $\St_n(R).$ Let $X' = \theta_n^{-1}((\prod_{(i', j')} \varepsilon_{i'j'})^{-1}) \cdot X$ Then, clearly 
$\theta_n(X') =  ((\prod_{(i', j')} \varepsilon_{i'j'})^{-1})\cdot \theta_n(X)  = I_n,$ i.e.,
$\theta_n(X') \in \K_{2, n}(R).$ We claim that $\eta(X') = Y(\widetilde{P_n}(R))_2$ i.e., 
$X'(\widetilde{P_n}(R))_2 = Y(\widetilde{P_n}(R))_2$ i.e., $X'Y^{-1} \in (\widetilde{P_n}(R))_2.$ 
We prove $\theta_n(X'Y^{-1}) \in (\EP_n(R))_2,$ from which it follows that 
$X'Y^{-1} \in (\widetilde{P_n}(R))_2,$ proving that $X'(\widetilde{P_n}(R))_2 = Y(\widetilde{P_n}(R))_2,$
as required. For this note that 
$\theta_n(Y) \in \EP_n(R)$ and write 
$\theta_n(Y) = 
\begin{pmatrix}
1 & 0 \\
v & (\theta_n(Y))_{\rm rd}
\end{pmatrix} = \begin{pmatrix} 
1 & 0 \\
v & I_{n-1}
\end{pmatrix} 
\begin{pmatrix} 
1 & 0 \\
0 & (\theta_n(Y))_{\rm rd}
\end{pmatrix}$ 

Then, 
$\theta_n(X'Y^{-1})$ 
$$= (\prod_{(i', j')} \varepsilon_{i'j'})^{-1}) \cdot \begin{pmatrix} 
1 & 0 \\
0 & (\theta_nY)_{\rm rd} 
\end{pmatrix} \cdot \begin{pmatrix} 
1 & 0 \\
0 & ((\theta_n(Y))_{\rm rd})^{-1}
\end{pmatrix}\cdot \begin{pmatrix} 
1 & 0 \\
-v & I_{n-1}
\end{pmatrix}.$$ Thus, 
$\theta_n(X'Y^{-1}) = \begin{pmatrix} 
1 & 0 \\
-v & I_{n-1}
\end{pmatrix}.$ That this clearly lies in $(\EP_n(R))_2,$ can be seen by writing it as a product of elementary generators 
of the type $\E_{i1}(v_i),$ where $v = (v_2, \cdots, v_n)^t.$ 

\item[{4(c)}] We check ${\rm ker}~\lambda \subset {\rm Im}\mu.$ Let $[Z](\K_{1,n-1}(R))_2 \in {\rm ker}~\lambda,$ 
i.e., $\begin{pmatrix} 
1 & 0 \\
0 & Z 
\end{pmatrix} \in \E_n(R),$ with $Z \in \K_{1, n-1}(R).$ Then, clearly $\mu(\theta_n^{-1}\widetilde{P_n}(R))_2 = 
\begin{pmatrix} 
1 & 0 \\
0 & Z 
\end{pmatrix},$ proving that ${\rm ker}~\lambda \subset {\rm Im}\mu.$

\end{itemize} 

\end{itemize} 
\end{proof} 

\begin{prop}\label{prop:final} 
Let $R$ be a ring and let $n \geq 3$ be an integer.Then, 
$$(\K_{2,n}(R))_2 \xrightarrow{\delta} \K_{2,n}(R) \xrightarrow{\eta} \pi_1(\EUm_n(R)) \xrightarrow{\mu} 
\K_{1,n-1}(R)/(\K_{1,n-1}(R))_2  \xrightarrow{\lambda} $$ 
$$\K_{1,n}(R) \xrightarrow{\alpha}  \pi_0(\Um_n(R)) \xrightarrow{\beta} \K^s_{0,n-1}(R) \xrightarrow{\gamma} \K^s_{0,n}(R).$$
is an exact sequence of pointed sets. 
\end{prop} 
\begin{proof} 
It remains to check exactness at $\K_{1,n}(R).$
\end{proof} 

We recall the following results regarding injective and surjective stability in the context of commutative 
rings. 

\begin{thm} \label{thm:inj-sur-stability} 
Let $A$ be a commutative and associative ring such that the maximal ideal space of $A$ is a noetherian space of dimension 
$\leq d$ (e.g. $A$ is a noetherian ring of Krull dimension atmost $d$), then: 
\begin{itemize} 
\item[{1.}] The maps $\GL_n(A)/\E_n(A) \rightarrow \K_1(A)$ are isomorphism of groups for all $n \geq d+3$ 
(and for all $n \geq 3$, if $d = 1.$) In particular, for all $n,m \geq d + 3$,  
$\K_{1,n}(R) \rightarrow \K_{1, m}(R)$ is an isomorphism (and for all $n,m \geq 3$, if $d = 1.$). 
The map 
$\K_{1, n}(R) \rightarrow \K_{1,n+1}(R)$ is injective for $n \geq d+2.$ 
(\cite{Bass} and \cite{Vas}.) 

\item[{2.}] For all $n \geq d+2$ the natural map 
$\K_{2, n}(A) \rightarrow \K_2(A)$ is surjective 
and the natural map $\K_{2, n+1}(A) \rightarrow \K_2(A)$ is an isomorphism. 
In particular, the map $\K_{2, n}(A) \rightarrow \K_{2, n+1}(A)$ is surjective 
for all $n \geq d+2.$ (Surjective stability results are of Keith Dennis and 
Vaserstein.) 
If $A$ is a Dedekind ring of 
arithmetic type with infinitely many units (as in Bass-Milnor-Serre \cite{BMS}), 
then the map $\K_{2, 2}(A) \rightarrow \K_{2}(A)$ is surjective 
and the map $\K_{2, 3}(A) \rightarrow \K_{2}(A)$ is bijective. In particular,
$\K_{2, 2}(A) \rightarrow \K_{2, 3}(A)$ is surjective. 
\end{itemize} 
\end{thm}

We deduce an important Corollary of the result above.

\begin{cor}\label{thm:trivial} 
Let $R$ be a commutative and associative ring such that the maximal 
ideal space of $R$ is a noetherian space of dimension $\leq d$. 
Then, $\pi_1(\EUm_n(R), e\EP_n(R)) = e,$ for all $n \geq d + 3.$ 
\end{cor}
\begin{proof}
Note that $n \geq d + 3$ implies that $\pi_0(\Um_n(R)) = e.$ This gives the following exact sequence of groups: 
$$(\K_{2,n}(R))_2 \rightarrow \K_{2,n}(R) \rightarrow \pi_1(\EUm_n(R)) \rightarrow \K_{1,n-1}(R)/(\K_{1,n-1}(R))_2  
\rightarrow $$ $$\K_{1,n}(R) \rightarrow 1.$$  

Noting that $(\K_{2,n}(R))_2$ contains the ${\rm image}(\K_{2,n-1}(R) \rightarrow \K_{2,n}(R))$ (see (\S 5));
we have $$\frac{\K_{2,n}(R)}{(\K_{2,n}(R))_2} \simeq  
\frac{\K_{2,n}(R)/{\rm image}(\K_{2,n-1}(R) \rightarrow \K_{2,n}(R))}
{(\K_{2,n}(R))_2/{\rm image}(\K_{2,n-1}(R) \rightarrow \K_{2,n}(R))}.$$

Thus, $$\frac{\K_{2,n}(R)}{(\K_{2,n}(R))_2} \simeq 
\frac{{\rm cokernel}(\K_{2,n-1}(R) \rightarrow \K_{2,n}(R))}{(\K_{2,n}(R))_2/{\rm image}(\K_{2,n-1}(R) \rightarrow \K_{2,n}(R))}.$$

This now gives the following exact sequence of groups: 
$$e \rightarrow \frac{{\rm cokernel}(\K_{2,n-1}(R) \rightarrow \K_{2,n}(R))}{(\K_{2,n}(R))_2/{\rm image}(\K_{2,n-1}(R) \rightarrow \K_{2,n}(R))} \rightarrow \pi_1(\EUm_n(R)) \rightarrow$$
$$ \frac{{\rm ker}(\K_{1,n-1}(R) \rightarrow \K_{1, n}(R))}{(\K_{1,n-1}(R))_2} \rightarrow e.$$ 

That $\pi_1(\EUm_n(R)) = e$ for all $n \geq d + 3$ now follows from 
Theorem \ref{thm:inj-sur-stability}. 
\end{proof}

\begin{cor} 
Let $A$ be a Dedekind ring of arithmetic type with infinitely many units.
Then, $\pi_1(\EUm_n(R), e\EP_n(R)) = e$ for all $n \geq 3.$ 
\end{cor} 
\begin{proof} 
Follows from the corresponding stability results for Dedekind rings of 
arithmetic type with infinitely many units. 
\end{proof}

\noindent
{\bf Acknowledgement:} The second author acknowledges the financial support of DAAD 
(Deutscher Akademischer Austausch Dienst) when this work was in progress. 
She thanks Prof. Ravi Rao, TIFR, Mumbai for his support and also thanks 
Prof. Dilip Patil, IISc, Bengaluru for useful suggestions. 
She thanks her husband Shripad, for his constant and enthusiastic support and useful mathematical discussions. 
\vskip5mm 
She takes this opportunity to wish the first author on his seventieth birthday and hopes he has a healthy life ahead.


\begin{thebibliography}{10}
\bibitem{Bak} Bak, A.; Brown, R.; Minian, G.; Porter, T.
 Global actions, groupoid atlases and applications.  
J. Homotopy Relat. Struct.  1,  no. 1, 101--167 (electronic), 2006. 
\bibitem{Bass} Bass, H. Algebraic $K$-theory. W. A. Benjamin, Inc., 
New York-Amsterdam, 1968.
\bibitem{Bass1} Bass, H. K-theory and stable algebra. Publ. I.H.E.S., 22, 489--544, 1964. 
\bibitem{BMS} Bass, H., Milnor J., Serre J.-P., Solution of the congruence subgroup problem for ${\rm SL}_{n}$, $(n\geq 3)$ and ${\rm Sp}_{2n}$, $(n\geq 2)$.
Inst. Hautes {\'E}tudes Sci. Publ. Math. No. 33, 59--137, 1967. 
\bibitem{Mil} Milnor, J. Introduction to algebraic K-theory, Annals of Mathematics Studies, Princeton Universities Press.  
\bibitem{Rao-vdk} Rao, Ravi A., van der Kallen, W.; Improved stability for $SK_1$ and $WMS_d$ of a non-singular affine algebra. 
$K$-theory (Strasbourg, 1992).  Ast{\'e}risque  No. 226, 11, 411--420, 1994.   
\bibitem{Sus-Tu} Suslin A. A., Tulenbayev, M. S.; A theorem on stabilization 
for Milnor's 
$K_2$ functor. Rings and Modules, Zap. Nau. Sem. Leningrad. Otdel. Mat. Inst. Steklov. 
(LOMI) 64, 131--152 (Russian), 1976. 
\bibitem{SV} Suslin, A. A., Vaserstein, L. N.; Serre's problem on
Projective Modules over Polynomial Rings and Algebraic K-theory, Math.
USSR Izvestija 10, 937--1001, 1976.
\bibitem{Vas} Vaserstein, L. N.; On the stabilization of the general linear 
group over a ring, Mat. Sb. (N.S.), 79(121):3(7), 405--424, 1969. 
\bibitem{vdK1} van der Kallen, W.; A group structure on certain orbit 
sets of unimodular rows, J. Algebra  82, no. 2, 363--397, 1983.
\bibitem{vdK2} van der Kallen, W.; A module structure on certain orbit sets of
unimodular rows, J. Pure Appl. Algebra 57, 657--663, 1975.

\end{thebibliography}
\end{document}